\documentclass[10pt]{paper}%
\usepackage[T1]{fontenc}
\usepackage{amsmath,amssymb}
\usepackage{fullpage}
\usepackage{youngtab}
\usepackage{amsthm}
\usepackage{url}
\usepackage{amsmath}
\usepackage{graphicx}
  \usepackage[all]{xy}
\usepackage{amsfonts}
\usepackage{amssymb}%
\newcommand{\ulambda}{{\boldsymbol{\lambda}}}
\newcommand{\umu}{{\boldsymbol{\mu}}}
\newcommand{\unu}{{\boldsymbol{\nu}}}

\newcommand{\bs}{{\boldsymbol{s}}}
\newcommand{\uemptyset }{{\boldsymbol{\emptyset}}}
\newcommand{\Uglov}[1]{\Phi_{#1}}

\newcommand{\Usli}{\mathcal{U}_v (\mathfrak{sl}_{\infty})}
\newcommand{\Uslaff}{\mathcal{U}_v (\widehat{\mathfrak{sl}}_{e})}

\newtheorem{Th}{Theorem}[section]
\newtheorem{lemma}[Th]{Lemma}

\newtheorem{Prop}[Th]{Proposition}

\theoremstyle{remark}
\newtheorem{Rem}[Th]{Remark}{\rmfamily}
\theoremstyle{definition}
\newtheorem{Def}[Th]{Definition}{\rmfamily}
\newtheorem{Exa}[Th]{Example}{\rmfamily}

\newtheorem{abs}[Th]{\bfseries}

\begin{document}

\title{On the regularization process for Ariki-Koike algebras }
\author{N. Jacon}
\maketitle
\date{}

\begin{abstract}
The aim of this note is to study a generalization of theorems  by James and Fayers on the modular representations of the  symmetric group  and its Hecke algebra 
 to the case of the complex reflection groups of type $G(l,1,n)$ and the associated Ariki-Koike algebra.  
\end{abstract}

\section{Introduction}

One of the main and still open problem in the representation theory 
 of finite groups is the explicit determination of the irreducible representations for the symmetric groups $\mathfrak{S}_n$, for $n\in \mathbb{N}$,  over a field of  characteristic $p>0$. 
 The main informations on these representations are contained in a fundamental object: the decomposition matrix.  
  By the works of James, the problem of computing this matrix  may be attacked using the representation theory of Hecke algebras.  
   Indeed, James' conjecture 
   predicts that, in the case where $p^2>n$, the  decomposition matrix of the symmetric group corresponds to the decomposition matrix 
 of a non semisimple deformation of the group algebra $\mathbb{C}\mathfrak{S}_n$: the Hecke algebra. Both matrices have their 
  rows labelled by the set of partitions of rank $n$ (which itself labels the set of simple modules of $\mathbb{C}\mathfrak{S}_n$) where as their
   columns are labelled by a certain subset of partitions called the set of $p$-regular partitions (which itself labels the set of simple modules of 
     $\mathbb{F}\mathfrak{S}_n$ where $\textrm{car} (\mathbb{F})=p$). 
  
An algorithm for the computation of the decomposition matrices for Hecke algebras over $\mathbb{C}$ is available. This algorithm comes from a conjecture 
 given by Lascoux, Leclerc and Thibon \cite{LLT}   and proved by Ariki \cite{ariki}. It asserts that the decomposition matrix of the Hecke algebra  over $\mathbb{C}$ 
  is given by the evaluation at $v=1$ of the matrix of the canonical basis for the basic representation of the quantum algebra $\Uslaff$. 
 Recently, several authors have shown that the matrix of the 
 canonical basis itself (and not only its specialization at $v=1$) admits 
  an interpretation in terms of  ``graded representation theory'' of Hecke algebras (see \cite{Kl} and the references therein).  This matrix 
   can thus be called the $v$-decomposition matrix of the Hecke algebra. 
  
  A nice general property on these matrices has been revealed  by  James  \cite{Ja} (for the decomposition matrix of the symmetric group) 
   and Fayers \cite{F2}  (for the $v$-decomposition matrix of the Hecke algebra).   One can explicitly   associate to each 
    partition $\lambda$ of $n$ a certain $p$-regular partition, the {\it regularization} of $\lambda$,
       and computes the associated decomposition number.  This result  brings in a very important partial order
        on the set of partitions which appears in many ways in the representation theory of the symmetric group: the dominance order.
  
This note is concerned with a generalization of the Hecke algebra of $\mathfrak{S}_n$: the Ariki-Koike algebra. 
 We present an analogue of James and Fayers' results for the decomposition matrices of these algebras and for  the matrices of the canonical bases for the irreducible highest weight $\Usli$-modules.  The concept of partitions
  is here replaced by the concept of ``multipartitions'', the set of $p$-regular partitions by the set of so called  cylindric multipartitions and the 
   dominance order on partitions with the dominance order on multipartitions. 
 The main results, Theorem \ref{princ} and Theorem \ref{princ1} give the desired analogues of James and Fayers' Theorems. 
The paper will be organized as follows. In the first section, we introduce the main combinatorial objects we will use in this paper: multipartitions and symbols and 
 present some useful properties on them. Then we define the notion of regularization of multipartitions. The third part is devoted to a 
  brief exposition  of the representation theory of $\Usli$ using the theory of Fock spaces.  All these notions are then used in the two last parts to obtain our main results.

\section{Multipartitions and symbols}

In this part, we give the combinatorial definitions which are needed for presenting our main results. 

\begin{abs} Let $l\in \mathbb{N}_{>0}$. We denote 
$$ \mathcal{S}^l:=\left\{ (s_0,\ldots,s_{l-l})\in \mathbb{Z}^l,\ |\ \forall i\in \{0,\ldots, l-2\}, s_{i+1}\geq s_i\right\}.$$ 
  Recall that a {\it partition} $\lambda=(\lambda_1,\ldots,\lambda_r)$ of $n\in \mathbb{N}$
 is an ordered  sequence of weakly decreasing non negative integers such that $|\lambda|:=\sum_{1\leq i\leq r} \lambda_i=n$. The  integer $h_{\lambda}:=\textrm{max} (i\geq 0\ |\ \lambda_i\neq 0)$ is called the {\it height } 
  of $\lambda$ with the convention that $h_{\lambda}=0$ if and only if $\lambda$ is $\emptyset$, the empty partition. 
A {\it multipartition} or {\it $l$-partition} of $n$ is an $l$-tuple  of partitions 
${\ulambda }
=(\lambda ^{0},\ldots ,\lambda ^{l-1 })$  such that, for  each $i\in \{0,1,\ldots,l-1\}$, $\lambda^i$ is a partition of rank $n_i\in \mathbb{N}$ and $\sum_{0\leq i\leq l-1} n_i=n$.  If $\ulambda$ is a multipartition of  $n$, we denote ${\boldsymbol{\lambda}}\vdash_{l}n$.  The height of $\ulambda$ is the non negative integer:
$$h_{\ulambda}:=\textrm{max} (h_{\lambda^{0}},\ldots,h_{\lambda^{l-1}}) , $$ 
and we have $h_{\ulambda}=0$ if and only if $\ulambda$ is the empty multipartition, which is  denoted by $\uemptyset$.  
 The dominance order on multipartitions is defined as follows. Let $\ulambda:= (\lambda^0,\ldots,\lambda^{l-1})$ 
  and $\umu:= (\mu^0,\ldots,\mu^{l-1})$  be two partitions of $n$ then we denote:
  $$\ulambda \unrhd \umu \iff \forall c\in \{0,\ldots,l-1\},\  \forall k\in \mathbb{N},\ \sum_{0\leq i < c} |\lambda^i| +\sum_{1\leq j\leq k} \lambda^{c}_j 
   \geq  \sum_{0\leq i <c} |\mu^i|\ +\sum_{1\leq j\leq k} \mu^{c}_j ,$$
where the partitions are considered with an infinite number of empty parts. 

\end{abs}


  \begin{abs} Let ${\bf s}=(s_0,\ldots, s_{l-1})\in \mathcal{S}^l$. 
We now define the notion of  shifted symbol.
 Following \cite[\S 5.5.5]{GJ},  let $\beta=(\beta_1,\ldots ,\beta_k)$ be a
  sequence of 
    integers and let $t$ be a positive  integer.  
 We set
 $$\beta (s):=(0,1,\ldots ,t-1,\beta_1+t,\ldots ,\beta_k+t).$$
 It is a sequence of rational numbers with exactly $k+t$ elements. 
For   $i=0,1,\ldots,l-1$, let $h^{i}$ be the height of the
partitions $\lambda^{i}$.  We consider the following sequence of rational numbers:
    $$\beta^{i}=(\lambda^{i}_{h^{i}}-h^{i}+h^{i},\ldots ,\lambda^{i}_j-j+h^{i},   \ldots ,\lambda^{i}_1-1+h^{i}).$$
 This is a sequence of strictly increasing integers if and only if $\lambda^i$ is a partition.  
Now, for $i=0,1,\ldots ,l-1$, we put
 $$hc^i =h^{i}-s_{i}
\textrm{ and }
hc^{\ulambda}=\textrm{max}(hc^{0},\ldots ,hc^{l-1}).$$
   Let $h$ be an integer such that $h\geq hc^{\ulambda}+1$.  The  {\it shifted ${\bf s}$-symbol} of $\ulambda$ of size $h$ is  the family of sequences 
$$\mathfrak{B}_{({\bf m},h)} (\ulambda)=(\mathfrak{B}^{0},\ldots ,\mathfrak{B}^{l-1}),$$
such that for $j=1,\ldots ,l$
$$\mathfrak{B}^{j}=(\beta^{j}(h-hc^j )).$$
Each sequence $\mathfrak{B}^{j}$ contains exactly $h+s_j$ elements $(\mathfrak{B}^{j}_{h+s_j},\ldots ,\mathfrak{B}^{j}_1)$.

  \end{abs}
  
\begin{abs}\label{form}   A shifted symbol  is usually represented as (and identified with) an $l$-row tableau where the $c$-th row (starting from the bottom) is $\mathfrak{B}^{c}$ (see \cite[\S 5.5.5]{GJ}). It  is written as follows:
 $$
\mathfrak{B}_{({\bf s},h)}  (\ulambda)=\left(
\begin{array}
[c]{lllll}
\mathfrak{B}^{l-1}_{h+s_{l-1}} &   \ldots  &  \ldots   &   \mathfrak{B}^{l-1}_2  & \mathfrak{B}^{l-1}_1 \\
\mathfrak{B}^{l-2}_{h+s_{l-2}} &   \ldots  &  \ldots   &  \mathfrak{B}^{l-2}_1 &\\     
\ldots &   \ldots  &  \ldots   &  \\  
\mathfrak{B}^{0}_{h+s_{0}}  &   \ldots  &  \mathfrak{B}^{0}_{1}    &    &   \\  
\end{array}
\right). $$
 In particular,  the $i^{\textrm{th}}$ column (starting from the right) of $\mathfrak{B}_{({\bf s},h)} (\ulambda)$ 
  contains exactly $l-c(i)$ elements where $c(i)=\textrm{min} (k\in \{0,1,\ldots,l-1\}\ |\ i+s_{k}-s_{l-1}>0)$ and this column is given as follows:
  $$
\left(
\begin{array}
[c]{c}
\mathfrak{B}^{l-1}_{i} \\
\mathfrak{B}^{l-2}_{i+s_{l-2}-s_{l-1}} \\  
\ldots \\
\mathfrak{B}^{c(i)}_{i+s_{c(i)}-s_{l-1}}  \\
\end{array}
\right).$$

 It is easy to recover the multipartition ${\boldsymbol{\lambda}}$  from the datum of an arbitrary shifted symbol. 
 Similarly, 
one can also easily recover  ${\bf s}\in \mathbb{Z}^l$  modulo a translation by an element $(x,\ldots,x)\in \mathbb{Z}^l$.

  \end{abs}

\begin{Exa}\label{e1}
With ${\boldsymbol{\lambda}}=(3,2.2.2,2.1)$ and $\mathbf{s}=(0,0,2)$, if we take $h=5$ we
obtain
\[
\mathfrak{B}_{((0,0,2),5)} (\ulambda)=\left(
\begin{array}
[c]{rrrrrrr}%
 0 & 1 & 2 & 3 & 4 & 6 & 8\\
0  & 1 & 4 & 5 & 6 &   \\
 0 & 1 & 2 & 3 & 7 & 
\end{array}
\right).
\]
With ${\bf s}=(0,1,2,3)$ and ${\boldsymbol{\lambda}}=(3.1,1.1,2.1.1,3)$ and $h=3$, the shifted symbol is as follows:
\[
\mathfrak{B}_{((0,1,2,3),3)} (\ulambda)=\left(
\begin{array}
[c]{rrrrrrr}%
 0 & 1 & 2 & 3 & 4 & 8 \\
0  & 1 & 3 & 4 & 6    \\
 0 & 1 & 3 & 4 &&\\
 0 & 2 & 5 &  && 
\end{array}
\right).
\]

\end{Exa}

\section{Regularization of multipartitions}
In this section, after having fixed an element in  $\mathcal{S}^l$,   we associate to each multipartition another one which belongs to a certain class of multipartitions: 
 the cylindric multipartitions. 
 
\begin{Def}\label{standard}  Assume that ${\bf s}\in \mathcal{S}^l$.
  A shifted symbol  $\mathfrak{B}_{({\bf s},h)}(\ulambda)$ 
 is called {\it standard} if and only in  each column of $\mathfrak{B}_{({\bf s},h)} (\ulambda)$, the 
 numbers weakly decrease from top to bottom. 
\end{Def}
\begin{Exa}
 The first symbol in Example \ref{e1} is not standard where as the second is. Let now $\mathbf{s}=(0,1,2)$ and 
   ${\boldsymbol{\lambda}}=(3.1,2.2.1,2.1)$.  If we take $h=5$ we
obtain
\[
\mathfrak{B}_{((0,1,2),5)} (\ulambda)=\left(
\begin{array}
[c]{rrrrrrr}%
 0 & 1 & 2 & 3 & 4 & 6 & 8\\
0  & 1 & 2 & 4 & 6 & 7  \\
 0 & 1 & 2 & 4 & 7 & 
\end{array}
\right)
\]
This symbol is standard. 
\end{Exa}

\begin{Def}\label{FLOTW}  Assume that  $\bs=(s_0,\ldots,s_{l-1}) \in \mathcal{S}^l$ then the multipartition ${\ulambda }%
=(\lambda ^{0},\ldots ,\lambda ^{l-1})$  is called  {\it cylindric} if  for every $c=0,\ldots ,l
-2$ and  $i\geq1$,  we have $\lambda _{i}^{c}\geq \lambda _{i+s_{c+1}-s_{c}}^{c+1}$  (the partitions are taken with an infinite number of empty parts). 
We denote by  $\Uglov{{\bf s}}$ the set of  {\it cylindric multipartitions} associated to ${\bs}\in \mathcal{S}^l$ and by 
$\Uglov{\bf s} (n)$ the set of cylindric  multipartitions of rank $n$. 
\end{Def}

\begin{Prop}\label{cy}
Let ${\bf s}\in \mathcal{S}^l$,  $\ulambda\vdash_l n $ and let  $\mathfrak{B}_{({\bf s},h)}  (\ulambda)$ be an associated shifted symbol. 
Then $\ulambda$  is cylindric if and only if  $\mathfrak{B}_{({\bf s},h)}  (\ulambda)$ is standard. 
\end{Prop}
\begin{proof} 
For all relevant $i\in \mathbb{N}$ and $c\in \{0,1,\ldots,l-2\}$, we have:
$$\lambda^{c}_i -i +s_c \geq \lambda^{c+1}_{i+s_{c+1}-s_c}-(i+s_{c+1}-s_c)+s_{c+1} \iff \lambda^c_i \geq \lambda^{c+1}_{i+s_{c+1}-s_c},$$
whence 
$$\mathfrak{B}^{c}_{i}\geq \mathfrak{B}^{c+1}_{i+s_{c+1}-s_c} \iff \lambda^c_i \geq \lambda^{c+1}_{i+s_{c+1}-s_c},$$
which is exactly what is needed to prove the assertion.
\end{proof}

\begin{abs}\label{reg} We now explain the process of regularization for multipartitions.  Let $\ulambda\vdash_l n$ 
and let $\bs=(s_0,\ldots,s_{l-1}) \in \mathcal{S}^l$. Let $\mathfrak{B}_{({\bf s},h)}  (\ulambda)$ be a shifted symbol associated to 
 $\ulambda$.  Then for each column of the symbol, we reorder the elements so that it is weakly decreasing from top to bottom.
  For example, the symbol in Example \ref{e1}
  $$\mathfrak{B}_{((0,1,1),4)} (\ulambda)=\left(
\begin{array}
[c]{rrrr}%
 0 & 1 & 3 & 9 \\
0  & 4 & 5 & 7  \\
 0 & 2& 3
\end{array}
\right)$$
 becomes: 
   $$\left(
\begin{array}
[c]{rrrr}%
 0 & 1 & 3 & 9 \\
0  & 2& 3 & 7  \\
 0 & 4& 5
\end{array}
\right).$$
  We claim that the resulting set is   a well defined standard symbol of a multipartition $\mathfrak{B}_{({\bf s},h)}  (\umu)=( \mathfrak{B}^0,\ldots,\mathfrak{B}^{l-1}   )$. Indeed, denote 
   this new set by $\mathcal{B}=(\mathcal{B}^0,\ldots,\mathcal{B}^{l-1})$ and assume that there exist 
   $c\in \{0,1,\ldots,  l-1\}$ and $j\in \mathbb{N}$ such that 
   $$\mathcal{B}^c_{j+1}\geq \mathcal{B}^{c}_{j}.$$
   Let $\mathcal{A}_1$ be the multiset of elements in the column of $\mathcal{B}^{c}_{j}$ which are less or equal  than $\mathcal{B}^{c}_{j}$ in $\mathfrak{B}_{({\bf s},h)}  (\ulambda)$. 
    Let $\mathcal{A}_2$ be the multiset of elements in the column of $\mathcal{B}^{c}_{j+1}$ which are greater or equal than $\mathcal{B}^{c}_{j}$ in $\mathfrak{B}_{({\bf s},h)}  (\ulambda)$.  Assume that  the column containing $\mathcal{B}^{c}_{j+1}$ contains $m$ elements.  We know that in the columns of 
     $\mathcal{B}$, the numbers are weakly decreasing from top to bottom. 
     Thus,  by the  construction of $\mathcal{B}$ and the above assumption, 
         we have $\sharp \mathcal{A}_1 \geq l-c$ and $\sharp \mathcal{A}_2 \geq m-l+c+1$. Now $\mathfrak{B}_{({\bf s},h)}  (\ulambda)$
  is a well defined symbol   so each row contains elements which are strictly increasing from left to right.  
  This implies that 
   the rows containing the elements of $\mathcal{A}_1$  in the column of    $\mathfrak{B}_j^c$ and the rows containing the elements of 
  $\mathcal{A}_2$  in the column of $\mathfrak{B}_{j+1}^c$ in $\mathfrak{B}_{({\bf s},h)}  (\ulambda)$ must be disjoint. 
      So the sum 
 $\sharp \mathcal{A}_1+\sharp \mathcal{A}_2$ must be less or equal than $m$. This is not the case, so the result follows. 
 
 \end{abs}

\begin{abs} Let   $\ulambda\vdash_l n$  and let $\mathfrak{B}_{({\bf s},h)}  (\ulambda)=( \mathfrak{B}^0,\ldots,\mathfrak{B}^{l-1}   )$ be an associated shifted symbol. 
For all $c_1\in \{0,1,\ldots,l-1\}$ and $j_1\in \{ 1,2,\ldots, h+s_{c_1}\}$, we set 
$$R(\ulambda)_{(j_1,c_1)}=\sharp \left\{ c\in \{0,1,\ldots,l-1\}\ |\ c> c_1,\ \mathfrak{B}_{j_1}^{c_1}<\mathfrak{B}^c_{j_1+s_{c}-s_{c_1}},\ \mathfrak{B}^{c_1}_{j_1}\notin \mathfrak{B}^c \right\},$$ 
and 
$$R(\ulambda)=\sum_{0\leq c \leq l-1} \sum_{1 \leq j_1 \leq h+s_c}  R(\ulambda)_{(j_1,c_1)}.$$
Clearly, this number does not depend of the choice of $h$ (and thus on the choice of the shifted symbol). 
\end{abs}
\begin{Def}
Let $\ulambda$ be a multipartition 
and let $\bs=(s_0,\ldots,s_{l-1}) \in \mathcal{S}^l$. Let $\mathfrak{B}_{{\bf s},h} (\ulambda)$ be the associated shifted symbol.
 There exists a cylindric multipartition $\ulambda^R$ such that  $\mathfrak{B}_{{\bf s},h} (\ulambda^R)= \mathfrak{B}_{{\bf s},h} (\ulambda)^R$. 
This multipartition is called the {\it regularization} of $\ulambda$.
\end{Def}

By construction, the regularization of a multipartition  is a cylindric multipartition by Proposition \ref{cy}. It is also clear that the regularization of a cylindric multipartition
 is itself. 

\begin{Exa}
Consider the $3$-partition $\ulambda:=(\emptyset,\emptyset,6.2)$. We set ${\bf s}=(0,1,1)$, $h=3$. Then we have:
\[
\mathfrak{B}_{((0,1,1),3)} (\ulambda)=\left(
\begin{array}
[c]{rrrr}%
 0 & 1 & 4 & 9  \\
0  & 1 & 2 & 3    \\
 0 & 1 & 2 
\end{array}
\right).
\]
 We obtain 
\[
\mathfrak{B}_{((0,1,1),3)} (\ulambda)=\left(
\begin{array}
[c]{rrrr}%
 0 & 1 & 2 & 3  \\
0  & 1 & 2 & 9    \\
 0 & 1 & 4 
\end{array}
\right).
\]
We have:
$$R(\ulambda)_{(j,c)}=\left\{
\begin{array}{lr}
0 & \textrm{ if }(j,c)\notin\{(1,0),(2,1),(1,1)\}\\
1 & \textrm{ if }(j,c)\in\{(1,0),(2,1),(1,1)\}\\
\end{array}\right.$$
and so we have  $R(\ulambda)=3$ . We can check that  $\ulambda^R=(2,6,\emptyset)$. 
\end{Exa}
\begin{Exa}
Consider the $3$-partition $\ulambda:=(5,\emptyset,2.1)$. We set ${\bf s}=(0,1,1)$, $h=3$. Then we have:
\[
\mathfrak{B}_{((0,1,1),3)} (\ulambda)=\left(
\begin{array}
[c]{rrrr}%
 0 & 1 & 3 & 5  \\
0  & 1 & 2 & 3    \\
 0 & 1 & 7
\end{array}
\right).
\]
We obtain 
\[
\mathfrak{B}_{((0,1,1),3)} (\ulambda)=\left(
\begin{array}
[c]{rrrr}%
 0 & 1 & 2 & 3  \\
0  & 1 & 3 & 5    \\
 0 & 1 & 7 
\end{array}
\right).
\]
and we have  $R(\ulambda)=R (\ulambda)_{(2,1)}=1$. Note that this symbol is standard and we have $\ulambda^R=(5,2.1,\emptyset)$.

\end{Exa}

%
%
\section{Action of $\Usli$ on the Fock space}

\begin{abs}
 Let $v$ be an indeterminate and let $\Usli$ be the enveloping algebra of $\mathfrak{sl}_{\infty}$ with Chevalley generators $e_j$, $f_j$ and $t_j$ ($j\in \mathbb{N}$), see for example \cite[\S 6.1]{GJ}.
   The simple roots and fundamental weights are denoted by $\alpha_k$ and $\Lambda_k$  for $k\in \mathbb{N}$ respectivly. Let now ${\bf s} \in \mathcal{S}^l$ and let 
    $\mathcal{F}^{\bf s}$ be the associated Fock space. This  is the $\mathbb{Q}(v)$-vector space defined as
follows:
\[
\mathcal{F}_{\mathbf{s}}=\bigoplus_{n\in\mathbb{Z}_{\geq0}}\bigoplus
_{{\boldsymbol{\lambda}}\vdash_{l}n}\mathbb{Q}(v){\boldsymbol{\lambda}}.
\]
One can define an action of $\Usli$ which turns     $\mathcal{F}_{\bf s}$  into an integrable  $\Usli$-module.
\end{abs}
\begin{abs}\label{f1}
Let   $\beta:=(\beta_1,\ldots,\beta_m)$  be a sequence of strictly increasing positive numbers.
 We write  $j\in \beta$ if the number $j$ appears in the sequence $\beta$ and, 
and   $j\notin \beta$ otherwise. In this case,  we write   $\beta \cup \{j\}$ for 
  the sequence of strictly positive numbers obtained by inserting the number $j$ in $\beta$.  Similarly, if 
 $j\in \beta$ then     $\beta \setminus \{j\}$ is defined to be the sequence obtained from $\beta$ by removing $j$. 

Let $\ulambda\vdash_l n$ and let 
$\mathfrak{B}_{({\bf s},h)} (\ulambda)$ be an associated shifted symbol. Let 
$\umu\vdash_l n$ and let 
$\mathfrak{B}_{({\bf s},h)} (\umu)$ be an associated shifted symbol (so we assume that $h\geq \textrm{max}(hc^{\ulambda},hc^{\umu} )+1$.) 
Then we write 
$$\ulambda \overset{j}{\underset {c}{\longrightarrow}} \umu $$
if for all $d\in \{0,1,\ldots,l-1\}$, we have 
$$\mathfrak{B}_{{\bf s},h}(\umu)^d=\left\{ \begin{array}{ll}   
 \mathfrak{B}_{{\bf s},h}(\ulambda)^d & \textrm{ if $d\neq c$}\\
(\mathfrak{B}_{{\bf s},h} (\ulambda)^d \setminus \{j+h-1\}))\cup{\{j+h\}} & \textrm{ if $d=c$}
\end{array}
   \right.$$
 Let $k\in \mathbb{N}$,   we  write:
   $$\ulambda \overset{j: k}{\underset {(c_1,\ldots,c_k)}{\longrightarrow}} \umu,$$
   if there exists a sequence of multipartitions
   $$\ulambda:=\ulambda[1],\ldots, \ulambda[k],\ulambda[k+1]:=\umu,$$
   such that for all $i=1,\ldots,k$, we have 
   $$\ulambda[k] \overset{j}{\underset {c_k}{\longrightarrow}} \ulambda[k+1].$$
   Hence, we have:
   $$\mathfrak{B}_{{\bf s},h}(\umu)^d=\left\{ \begin{array}{ll}   
 \mathfrak{B}_{{\bf s},h}(\ulambda)^d & \textrm{ if $d\neq c_i$ for all $i=1,\ldots,k$}\\
(\mathfrak{B}_{{\bf s},h} (\ulambda)^d \setminus \{j+h-1\}))\cup{\{j+h\}} & \textrm{ if $d=c_i$ for  $i=1,\ldots,k$ }
\end{array}
   \right.$$
   We also write $\ulambda \overset{j: k}{\longrightarrow} \umu $ if there exists  a sequence ${(c_1,\ldots,c_k)}\in \{0,1,\ldots,l-1\}^k$ such that 
   $\ulambda \overset{j: k}{\underset {(c_1,\ldots,c_k)}{\longrightarrow}} \umu $. 
   Finally, given $\ulambda$ and $\umu$ such that $\ulambda \overset{j: k}{\underset {(c_1,\ldots,c_k)}{\longrightarrow}} \umu $,  we define the following number:
   $$N_j (\ulambda,\umu)=\sum_{1\leq i\leq k} \left(\begin{array}{l} 
   \{\textrm{number of integers equals to $j+h-1$ in $\mathfrak{B}_{({\bf s},h)} (\umu)^c$ with $c\leq c_i$}\}   \\
   -  \{\textrm{number of integers equals to $j+h$ in $\mathfrak{B}_{({\bf s},h)} (\ulambda)^c$ with $c\leq c_i$}\}
   \end{array}\right)$$
  \end{abs}
  \begin{Exa}
  Consider the $4$-partition $\ulambda=(3.1,\emptyset,\emptyset,6.2)$ and ${\bf s}=(0,0,1,1)$. For $h=3$, we get the symbol:
  \[
\mathfrak{B}_{((0,0,1,1),3)} (\ulambda)=\left(
\begin{array}
[c]{rrrr}%
 0 & 1 & 4& 9  \\
0  & 1 & 2 & 3    \\
 0 & 1 & 2 \\
 0 & 2 & 5 \\ 
\end{array}
\right).
\]
  We have 
  $$(3.1,\emptyset,\emptyset,6.2)   \overset{3: 2}{\underset {(0,1)}{\longrightarrow}}  ( 3.2,1,\emptyset,6.2) ,$$
  where the symbol of $\umu$ is:
    \[
\mathfrak{B}_{((0,0,1,1),3)} (\umu)=\left(
\begin{array}
[c]{rrrr}%
 0 & 1 & 4& 9  \\
0  & 1 & 2 & 3    \\
 0 & 1 & 3 \\
 0 & 3 & 5 \\ 
\end{array}
\right).
\]
  \end{Exa}

  \begin{abs}   
For our purpose, we only need to describe the action of the Chevalley generators $f_i$ for $i\in \mathbb{Z}$ 
  and their divided power    $f^{(r)}_j:=\displaystyle\frac{f^r_j}{[r]_v !} $, 
 for $r\in \mathbb{N}$. This is given as follows. Let $\ulambda\vdash_l n$, let $r\in \mathbb{N}$ and let $j\in \mathbb{Z}$, then we have 
 $$f_j^{(r)} .\ulambda =\sum_{\ulambda \overset{j: r}{\longrightarrow}\umu } q^{N_j (\ulambda,\umu)} \umu .$$
%
%
It is known that the $\Usli$-submodule $V({\bf s})$ of $\mathcal{F}^{\bf s}$ generated by the empty multipartition is 
 an irreducible highest weight module for $\Usli$.

\end{abs}

\begin{abs}
Let $\Usli^-$ be the subalgebra of $\Usli$ generated by the $f_i$'s. We have a ring automorphism $x\mapsto \overline{x}$ of  $\Usli^-$ 
 such that 
 $$\overline{f}_j=f_j\ (j\in \mathbb{N}),\qquad\textrm{ and } \qquad \overline{v}=v^{-1},$$
which induces a $\mathbb{C}$-linear map $v\mapsto \overline{v}$ on $V({\bf s})$ defined by:
$$\overline{v.\emptyset}=\overline{v}.\emptyset.$$
Using this, one can define the canonical basis elements of $V({\bf s})$. They are elements 
 of $V({\bf s})$ which are parametrized by the set of cylindric multipartitions 
 $$\{ b_{\ulambda}\ |\ \ulambda\in \Uglov{{\bf s}}\},$$
 and characterized by the following property:
 $$\forall \ulambda\in \Uglov{\bf s} ,\ \overline{b_{\ulambda}}= b_{\ulambda},\ b_{\ulambda}=\ulambda + \sum_{\umu\vdash_l n} d_{\umu,\ulambda} (v) \umu,$$
 for elements $d_{\umu,\ulambda} (v)\in v\mathbb{Q}[v]$ with $\umu\vdash_l n$ such that  $\ulambda\neq \umu$. 
\end{abs}

\begin{abs}

 We now briefly explain an algorithm for computing the canonical bases of irreducible highest weight $\Usli$-modules. 
 An analogue    of this algorithm has already been described in \cite{Algo} (and the proofs can be found therein) 
  in a more general setting but here it can be  simplified here. 
 Let $\ulambda=(\lambda^0,\ldots,\lambda^{l-1})\vdash_l n $ be a  non empty cylindric 
  multipartition  and let $\mathfrak{B}_{({\bf s},h)}:=\mathfrak{B}_{({\bf s},h)} (\ulambda)$ be an associated shifted symbol.  Denote
  $ \mathfrak{B}_{({\bf s},h)}=(\mathfrak{B}^0,\ldots,\mathfrak{B}^{l-1})$. 
   From this datum, we define a new multipartition $\ulambda^-\vdash_l n' <n$  together with two integers 
    $j(\ulambda)$ and $r(\ulambda)$. 
    
    \begin{enumerate}
    \item Let $c (\ulambda) $ be the minimal integer such that $\lambda^{c (\ulambda)}$ is non empty. This means that 
     there exists $i (\ulambda) \in \mathbb{N}$  such that 
     $$\mathfrak{B}^{c (\ulambda) }_{i (\ulambda)}>\mathfrak{B}^{c (\ulambda)}_{i (\ulambda)+1}+1.$$
     Assume that $i (\ulambda)$ is minimal with this property and set 
      $$j (\ulambda):=\mathfrak{B}^{c (\ulambda)}_{i (\ulambda)}.$$
    \item Let $r(\ulambda)\in \mathbb{N}$ be maximal  such that 
      $$\mathfrak{B}^{c(\ulambda)}_{i(\ulambda)}=\mathfrak{B}^{c(\ulambda)+1}_{i(\ulambda)+s_{c(\ulambda)+1}-s_{c(\ulambda)}}= \ldots = \mathfrak{B}^{c(\ulambda)+r(\ulambda)-1}_{i(\ulambda)+s_{c(\ulambda)+r(\ulambda)-1}-s_{c(\ulambda)}}=j (\ulambda),$$    
    then we have $r(\ulambda)\geq 1$ and because $\ulambda$ is cylindric, by Proposition \ref{cy}, we deduce that:
    $$        j(\ulambda)>  \mathfrak{B}^{c(\ulambda)+k-1}_{i(\ulambda)+s_{c(\ulambda)+k-1}-s_{c(\ulambda)}+1}+1$$
    for all $k=1,\ldots, r (\ulambda)$.
    \item  Take the symbol obtained by replacing in  $ \mathfrak{B}_{({\bf s},h)}$  all the elements 
  $$\mathfrak{B}^{c(\ulambda)}_{i(\ulambda)}=\mathfrak{B}^{c(\ulambda)+1}_{i(\ulambda)+s_{c(\ulambda)+1}-s_{c(\ulambda)}}= \ldots = \mathfrak{B}^{c(\ulambda)+r(\ulambda)-1}_{i(\ulambda)+s_{c(\ulambda)+r(\ulambda)-1}-s_{c(\ulambda)}}=j (\ulambda)$$    
  by $j (\ulambda)-1$. This is a well defined shifted symbol $\mathfrak{B}_{({\bf s},h)}'$. Thus there exists a unique $l$-partition $\ulambda^-\vdash_l n-r(\ulambda)$ 
  such that  $\mathfrak{B}_{({\bf s},h)} (\ulambda^-)=   \mathfrak{B}_{({\bf s},h)}'$. 
    \item By construction,  $\mathfrak{B}_{({\bf s},h)}'$  is standard so $\ulambda^-$ is cylindric.
    \end{enumerate}

Using this, we can  produce a sequence of  cylindric multipartitions $\ulambda[{k}]$ with $k=1,\ldots, m+1\in \mathbb{N}$ such that $\ulambda[1]=\ulambda$, 
 $\ulambda[m+1]=\uemptyset$ and $\ulambda[k+1]=\ulambda[k]^-$ for $k=2,\ldots,m$. 
  Set $j_k:=j(\ulambda [k])-h$ and $a_k:=r(\ulambda [k])$  for $k=1,\ldots,m$. Then note that we have:
 $$\ulambda[m+1] \overset{j_m: a_m }{\longrightarrow}  \ulambda[m] \overset{j_{m-1}: a_{m-1} }{\longrightarrow} \ldots 
 \overset{j_1: a_1 }{\longrightarrow}  \ulambda[1] .$$
We have  thus  defined two sequences of  integers $({j}_m,\ldots,{j}_1)$ and $(a_m,\ldots,a_1)$. 
Then   we can define 
$$a_\ulambda:= f^{(a_1)}_{j_1}  f^{(a_2)}_{j_2} \ldots f^{(a_m)}_{j_m}.\uemptyset . $$ 
  By induction, our construction implies that  we have:
$$a_\ulambda=\ulambda+\sum_{\ulambda\rhd \umu } b_{\umu,\ulambda} (v) \umu.$$
 (the  proof is exactly the same as   in \cite[Prop 4.6]{J}.) Note  that if  $b_{\umu,\ulambda} (v) \neq  0$ then the multiset of elements appearing in $\mathfrak{B}_{{\bf s},h} (\umu)$ is the same as the one of $\mathfrak{B}_{{\bf s},h} (\ulambda)$ (for $h$ large enough).  
When $\ulambda$ runs the set of all cylindric multipartitions, these elements provide  a basis for $V(\bs)$ and an algorithm for the computation of the canonical basis (see below or \cite[Ch. 6]{GJ}). 
  This, in turn, implies that  if  $d_{\umu,\ulambda} (v) \neq 0$ then the multiset of elements appearing in $\mathfrak{B}_{{\bf s},h} (\umu)$ is the same as the one of $\mathfrak{B}_{{\bf s},h} (\ulambda)$   (for $h$ large enough) and we have:
  $$b_{\ulambda}=\ulambda+\sum_{\ulambda\rhd \umu } d_{\umu,\ulambda} (v) \umu.$$
\end{abs}
\begin{Exa}
Let ${\bf s}=(0,0,1)$ and $\ulambda=(2.2,2.2,2.2.1)$, for $h=3$, we get the following symbol:
\[
\mathfrak{B}_{((0,0,1),3)} (\ulambda)=\left(
\begin{array}
[c]{rrrr}%
 0 & 2& 4 & 5  \\
0  & 3 & 4 &     \\
 0 & 3 & 4
\end{array}
\right)
\]
which is standard. We begin with $c (\ulambda)=0$ and $j (\ulambda)=4$. Then we have 
  $r(\ulambda)=2$. The new symbol is 
  \[
\left(
\begin{array}
[c]{rrrr}%
 0 & 2& 4 & 5  \\
0  & 2 & 4 &     \\
 0 & 2 & 4
\end{array}
\right)
\]
which is the symbol $\mathfrak{B}_{((0,0,1),3)} (\ulambda[2])$ with $\ulambda [2]=(2.1,2.1,2.2.1)$.   Then we have 
$c (\ulambda [2] )=0$,
$j(\ulambda [2])=4$, 
 $r(\ulambda[2])=3$ and we get the symbol:
  \[
\left(
\begin{array}
[c]{rrrr}%
 0 & 2& 3 & 5  \\
0  & 2 & 3 &     \\
 0 & 2 & 3
\end{array}
\right)
\]
which is the symbol $\mathfrak{B}_{((0,0,1),3)} (\ulambda[3])$ with $\ulambda [3]=(1.1,1.1,2.1.1)$. Continuing in this way, keeping the notation of the above remark,  we obtain 
$$\ulambda [4]=(1,1,2.1),\qquad  \ulambda [5]=(\emptyset,\emptyset,2),\qquad \ulambda [6]=(\emptyset,\emptyset,1),\qquad\ulambda [7]=\uemptyset.$$

\end{Exa}
\begin{abs}\label{llt}
Following \cite[\S 6.25]{mat}, the above construction provides an algorithm for the computation of the canonical basis of $V(\bs)$: The (modified) LLT algorithm. This is done recursively as follows. Let $n\in \mathbb{N}_{>0}$. For all $k<n$, assume that 
 we have constructed all the canonical basis elements 
 $$\{b_{\ulambda}\ |\ \ulambda\in \Uglov{\bs}(k)\}.$$
 Now, let $\ulambda\in  \Uglov{\bs}(n)$. We want to compute $b_{\ulambda}$. 
 \begin{enumerate}
 \item We set $c_{\ulambda}=f_{ j(\ulambda)-h}^{(r(\ulambda))} b_{\ulambda^-}$ and we have $\overline{c_{\ulambda}}=c_{\ulambda}$.
 \item The element $b_{\ulambda^-}$ is known by induction. Again, by the above construction,   we have 
 $$c_{\ulambda}=\ulambda+\sum_{\ulambda \rhd \unu} \widehat{d}_{\unu,\ulambda} (v) \unu,$$
 for Laurent polynomials $ \widehat{d}_{\unu,\ulambda} (v)$ with $ \widehat{d}_{\ulambda,\ulambda} (v)=1$ and we obtain 
 $$b_{\ulambda}=c_{\ulambda}-\sum_{\ulambda\rhd \unu } \alpha_{\unu,\ulambda} (v) b_\unu,$$
 for some Laurent polynomials  $ \alpha_{\unu,\ulambda} (v)$ such that  $\alpha_{\unu,\ulambda} (v)= \alpha_{\unu,\ulambda} (v^{-1}).$ 
 \item We find the greatest multipartition $\unu\neq \ulambda$ with respect to $\rhd$ 
 such that $\widehat{d}_{\unu,\ulambda} (v) \neq 0$. If no such multipartitions exist, then we have $c_{\ulambda}=b_{\ulambda}$.  
 Otherwise, $\alpha_{\unu,\ulambda} (v)$ is the unique bar invariant Laurent polynomial such that  the coefficients 
  of $\alpha_{\unu,\ulambda} (v)$ and $\widehat{d}_{\unu,\ulambda} (v)$ associated to the $v^i$ with $i\leq 0$ are the same. We then replace  $c_{\ulambda}$ 
    with the bar invariant element $c_{\ulambda}-\alpha_{\unu,\ulambda} (v) b_{\unu}$ and we  repeat the last  step until   all of the coefficients  $\widehat{d}_{\unu,\ulambda} (v)$ belong to $v\mathbb{Z}[v]$  for $\ulambda\neq \unu$. 
 
  \end{enumerate}

\end{abs}

\section{The main results}
The first main result of this paper is the following theorem concerning the computation of the coefficients of the canonical basis elements for irreducible highest weight $\Usli$-modules.. 
\begin{Th}\label{princ}
Assume that $\ulambda$ and $\umu$ are $l$-partitions of rank $n$ and assume that $\umu$ is cylindric. Then we have 
 $d_{{\ulambda},\umu} (v)=0$ unless $\umu\unrhd \ulambda^R$ 
 while $d_{{\ulambda},\ulambda^R} (v)=v^{R(\ulambda)}$.
\end{Th}
The strategy for the proof of this theorem is modeled on the one presented in  \cite{F2}. However, the proofs of the preparatory results we need here are of course different 
 than the ones given in  \cite{F2}.   
The first  lemma is the analogue of \cite[Lemma 2.3]{F2}. 

\begin{lemma}\label{l1}
Let $\ulambda=(\lambda^0,\ldots,\lambda^{l-1})\vdash_l n $ and $\umu=(\mu^0,\ldots,\mu^{l-1})\vdash_l n$.  Assume that 
 $\umu$ is cylindric and that  $\umu^- \overset{j: k}{\longrightarrow} \umu $. 
Let  $\unu$ be a partition of $n-k$ such that $\unu \overset{j:k}{\longrightarrow} \ulambda$ and $\umu^- \unrhd \unu^R$.
   Then $\umu \unrhd \ulambda^R$ with equality only if $\umu^-=\unu^R$. 

\end{lemma}
\begin{proof}
Assume that we have 
$$\mu^0=\ldots=\mu^{c (\umu)-1}=\emptyset \qquad \textrm{ and } \qquad \mu^{c (\umu)}\neq \emptyset.$$
Then, by construction, we have that 
$$\mathfrak{B}_{{\bf s},h} (\umu)^d=(\mathfrak{B}_{{\bf s},h} (\umu^-)^d \setminus \{j+h-1\}))\cup{\{j+h\}} $$
 for $d= c(\umu),\ldots, c(\umu)+k-1$. Similarly, by the definition of the regularization, we have that 
 $$\mathfrak{B}_{{\bf s},h} (\ulambda^R)^d=(\mathfrak{B}_{{\bf s},h} (\unu^R)^d \setminus \{j+h-1\}))\cup{\{j+h\}} $$
for elements $d\in \{ r_1,\ldots, r_k\}\subset \{0,\ldots,l-1\}$. As we have $\umu^- \unrhd \unu^R$, we deduce that:
$$(\nu^R)^0=\ldots=(\nu^R) ^{c (\umu)-1}=\emptyset.$$
We have to check  that 
  $r_i\geq c (\umu)$ for all $i\in \{1,\ldots, k\}$. So let us assume that we have $r_1< c(\umu)$ (without loss of generality). 
   Thus, we have 
   $$\mathfrak{B}_{{\bf s},h} (\unu^R)_1^{r_1}=0-1+s_{r_1}+h=j+h-1$$
  which implies that $j=s_{c_1}$. 
  Thus, we deduce that $j+h=s_{c_1}+h\in \mathfrak{B}_{{\bf s},h} (\umu)^{c(\umu)}$. As $j-1\notin 
    \mathfrak{B}_{{\bf s},h} (\umu)^{c(\umu)}$ by the construction of $\umu$, we deduce that 
     $i (\umu)>s_{c(\umu)}-s_{r_1}+1$. However this implies that 
     $ \mathfrak{B}_{{\bf s},h} (\umu)$ is not standard. Indeed, we have in this case 
     $ \mathfrak{B}_{{\bf s},h} (\umu)_{i(\umu)}^{c (\umu)}>  \mathfrak{B}_{{\bf s},h} (\umu)_{i(\umu)+s_{r_1}-s_{c(\umu)}}^{r_1} $. This a 
      contradiction because $\umu$ is cylindric.
     
     Assume now that we have the equality $\umu = \ulambda^R$, this implies that $ \{ r_1,\ldots, r_k\}=\{c(\umu),\ldots,c(\umu)+k-1\}$. Thus we also have 
     $\umu^-=\unu^R$ which concludes the proof.

%
%
%
%
%
%
%
%
%

\end{proof}

The second lemma is the analogue of \cite[Prop 2.5]{F2}. 

\begin{lemma}\label{l2}
Let $\ulambda$ be a cylindric multipartition and let $\umu$ be a multipartition such that
 $\umu^R=\ulambda$. Let  $\umu^-$ be a multipartition of $n-k$ such that $\umu^- \overset{j: k}{\longrightarrow} \umu $ for some $j\in \mathbb{Z}$ and $k\in \mathbb{N}$ 
  and such that  $(\umu^-)^R= \ulambda^-$ then 
    the coefficient of $\umu$ in $f_{j}^{(k)} \umu^-$ is $v^{R(\umu)-R(\umu^-)}$.
\end{lemma}

\begin{proof}
Consider a shifted symbol $\mathfrak{B}_{({\bf s},h)}( \umu^-) $.  Let 
 $(i_1, c_1)$, \ldots,  $(i_m, c_m)$  be all the elements of  $\mathbb{N}\times \{0,1,\ldots,l-1\}$ such that 
 $$\mathfrak{B}_{({\bf s},h)}( \umu^-)_{i_1}^{c_1}=\ldots= \mathfrak{B}_{({\bf s},h)}( \umu^-)_{i_m}^{c_m}=j+h-1.$$
 By hypothesis, we have $m\geq k$ and one can assume without loss of generality that:
 $$\mathfrak{B}_{({\bf s},h)}( \umu)_i^c=\left\{
 \begin{array}{ll}
j+h & \textrm{ if $(i_t,c_t)=(i,c)$ for one $t\in \{1,\ldots,k\}$}\\
\mathfrak{B}_{({\bf s},h)}( \umu^-)_i^c & \textrm{ otherwise}
 \end{array}\right.$$  
%
%
%
%
%
 First, we easily see that for all $(i,c)$ such that $\mathfrak{B}_{({\bf s},h)}( \umu^-)_i^c \notin \{j+h-1,j+h\} $ we have 
  $$R(\umu^-)_{(i,c)}= R(\umu)_{(i,c)}.$$
  Let now $(i,c)$ be such that 
   $$\mathfrak{B}_{({\bf s},h)}( \umu^-)_{i}^{c}=j+h-1.$$
   \begin{itemize}
   \item Assume that $c\neq c_t$ for all $t\in \{1,\ldots,k\}$. Then we have $\mathfrak{B}_{({\bf s},h)}( \umu)_{i}^{c}=j+h-1$. 
    We claim that we have:
  $$R(\umu)_{(i,c)}-R(\umu^-)_{(i,c)}=
   \{\textrm{number of integers $t\in \{1,\ldots, k\}$  such that  ${c}< {c}_t$}\}.  $$
  We need to show that if ${c}< {c}_t$ then we have 
  $$\mathfrak{B}(\umu^-)_{i+s_{c_t}-s_{c}}^{c_t}> \mathfrak{B} (\umu^-)_{i}^{c} .$$
  Assume to the contrary that $\mathfrak{B}(\umu^-)_{i+s_{c_t}-s_{c}}^{c_t}\leq \mathfrak{B}(\umu^-)_{i}^{c} , $
  then by the process of regularization, in the construction of  $\mathfrak{B}_{({\bf s},h)}( \ulambda^-)$, the 
   number $\mathfrak{B}_{({\bf s},h)}( \umu^-)_{i_t}^{c_t}=j+h-1$ is send to a row $c'$ of the symbol
    which is greater that the row containing $\mathfrak{B}_{({\bf s},h)}( \umu^-)_{i}^{c}=j+h-1$. 
     This is impossible by the construction of $\ulambda^-$.  
 \item Assume that $c=c_t$ for $t\in \{1,\ldots,k\}$. Then we have $\mathfrak{B}_{({\bf s},h)}( \umu)_{i}^{c}=j+h$. 
  We claim that we have 
     $$R(\umu)_{(i,c)}=R(\umu^-)_{(i,c)}.$$
     This comes from the fact that if there exist $c'>c$ such that $c'\neq c_s$ for $s\in \{1,\ldots,k\}$ and 
      $$\mathfrak{B}(\umu^-)_{i+s_{c'}-s_{c'}}^{c'}> \mathfrak{B} (\umu^-)_{i}^{c},$$
      then $j+h-1\notin \mathfrak{B}_{({\bf s},h)}( \umu^-)^{c'}$ which follows from the construction of $\ulambda^-$. 
   \end{itemize}
    Now assume that  $(i,c)$ is such that 
   $$\mathfrak{B}_{({\bf s},h)}( \umu^-)_{i}^{c}=j+h.$$
We follow the same kind of reasoning as above by showing that if 
$c_t>c$ then we have:
 $$\mathfrak{B}(\umu^-)_{i+s_{c_t}-s_{c_t}}^{c_t}>  \mathfrak{B} (\umu^-)_{i}^{c}.$$
 We deduce that:
  $$R(\umu^-)_{(i,c)}-R(\umu)_{(i,c)}=
   \{\textrm{number of integers $t\in \{1,\ldots, k\}$  such that  ${c}< {c}_t$}\}  $$
   Now, we use the formula in \S \ref{f1}  to deduce the result, we have:
   $$
   \begin{array}{rcl}
   N_j (\umu^-,\umu)&=&\displaystyle\sum_{1\leq i\leq k} \left(\begin{array}{l} 
   \{\textrm{number of integers equals to $j+h-1$ in $\mathfrak{B}_{({\bf s},h)} (\umu)^c$ with $c\leq c_i$}\}   \\
   -  \{\textrm{number of integers equals to $j+h$ in $\mathfrak{B}_{({\bf s},h)} (\umu^-)^c$ with $c\leq c_i$}\}
   \end{array}\right)\\
   &=&\displaystyle\sum_{ c\in \{0,\ldots,l-1\}  } 
   \{\textrm{number of integers $t\in \{1,\ldots, k\}$  s.t. ${c}< {c}_t$ and  $j+h-1\in \mathfrak{B}_{({\bf s},h)}( \umu)^{c}$    }\}   \\
   && -\displaystyle\sum_{ c\in \{0,\ldots,l-1\}  }   \{\textrm{number of integers $t\in \{1,\ldots, k\}$  s.t.  ${c}< {c}_t$ and  $j+h\in \mathfrak{B}_{({\bf s},h)}( \umu^-)^{c}$    }\} 
   \\
   &=& R (\umu)-R(\umu^-).
   \end{array}
   $$

\end{proof}
\begin{Rem}
 In \cite{LM}, B. Leclerc and H. Miyachi have given an explicit closed formula for the elements of the canonical bases for irreducible highest weight modules of level 
  $l=2$. This formula also uses the notion of symbols. When $l>2$, these element are more difficult to compute because their are non monomial in terms of the Chevalley operators 
   (contrary to the case $l=2$). It is easy to check that our formula for the above decomposition numbers are consistent with the ones of Leclerc-Myiachi in this particularly case.

\end{Rem}

\begin{abs}{\bf Proof of Theorem \ref{princ}.} The proof is exactly the same as in \cite[Thm 2.2]{F2}. We give it for the convenience of the reader.   We argue  induction on $n\in \mathbb{N}$ and on the dominance order. When $n=0$  or  when $\umu$ is minimal with respect to $\unrhd$, there is nothing to do. 
 So let us assume that $n>0$ and that we have a cylindric multipartition $\umu$ of rank $n$. By induction, the result holds when $\umu$ is replaced with $\umu^-$. 
Assume that we have  $\umu^- \overset{j: k}{\longrightarrow} \umu $ for $j\in \mathbb{Z}$ and $k\in \mathbb{N}$. Then we have:
$$f_j^{(k)} b_{\umu^-}=\umu+ \sum_{\umu\rhd \ulambda} \widehat{d}_{\ulambda,\umu} (v) \ulambda.$$
We will prove the property of the theorem for the numbers $\widehat{d}_{\ulambda,\umu} (v)$ and then deduce the result for the numbers ${d}_{\ulambda,\umu} (v)$. 
 Assume that $\widehat{d}_{\ulambda,\umu} (v)\neq 0$. Then there exists $\unu\vdash_{l} n-k$ such that $d_{\unu,\umu^-} (v)\neq 0$ and such that
  $\unu \overset{j: k}{\longrightarrow} \ulambda $. By induction, this implies that $\umu^-\unrhd \unu^R$.  We can then apply Lemma \ref{l1} which implies that 
  $\umu \unrhd \ulambda^R$. 
  
  Now let us assume that $\ulambda^R=\umu$. If $\unu\vdash_l n$  is such that $d_{\unu,\umu^-}(v)\neq 0$ and 
  $\unu \overset{j: k}{\longrightarrow} \ulambda $ then by induction, we have $\unu^R \unrhd \umu^-$. Then we obtain 
   $\unu^R=\umu^-$.  Thus, the coefficient $\widehat{d}_{\ulambda,\umu} (v)$ equals $\widehat{d}_{\unu,\unu^R} (v)$ 
 times the coefficient of $\ulambda$ in $f_j^{(m)} \unu$. By Lemma \ref{l2} and by induction,  this coefficient is $v^{R(\ulambda)}$.
  Hence, we have shown that $\widehat{d}_{\ulambda,\umu} (v)$ is zero unless  $\umu \unrhd \ulambda^R$ and 
   $\widehat{d}_{\ulambda,\ulambda^R} (v)= v^{R(\ulambda)}$.  The LLT algorithm in \S \ref{llt} implies that 
   $$d_{\ulambda,\umu} (v)=\widehat{d}_{\ulambda,\umu} (v)+\sum_{\xi \lhd \umu} \alpha_{\xi,\umu} (v) d_{\ulambda,\xi} (v)$$
   Now assume that $\umu$ is not greater than  $\ulambda^R$ with respect to the dominance order. This implies that for all 
    $\xi\vdash_l n$ such that $\xi \lhd \umu$, $\xi$ is not greater than $\ulambda$. By induction, we have 
     $d_{\lambda,\xi} (v)=0$. Thus we obtain $d_{\ulambda,\umu} (v)=\widehat{d}_{\ulambda,\umu} (v)$ and the result follows. This concludes the proof.

\end{abs}

\section{Regularization for Ariki-Koike algebras}
We now present consequences on the representation theory of Ariki-Koike algebras. 

\begin{abs}
Let  ${\bf s}=(s_0,\ldots,s_{l-1})\in \mathcal{S}^l$ and let  $\eta\in \mathbb{C}$.  We consider the associative $\mathbb{C}$-algebra  $\mathcal{H} (\bs)$  generated by $T_{0},\cdots ,T_{n-1}$ subject to the relations $(T_{0}-\eta ^{%
{s}_{0}})...(T_{0}-\eta ^{{s}_{l -1}})=0$, $(T_{i}-\eta
)(T_{i}+1)=0$, for $1\leq i\leq n$ and the type $B$ braid relations 
\begin{gather*}
(T_{0}T_{1})^{2}=(T_{1}T_{0})^{2},\quad T_{i}T_{i+1}T_{i}=T_{i+1}T_{i}T_{i+1}%
\text{ }(1\leq i<n), \\
T_{i}T_{j}=T_{j}T_{i}\text{ }(j\geq i+2).
\end{gather*}
$\mathcal{H} (\bs)$ is called {\it the Ariki-Koike algebra}.  Set $e=+\infty$ if $\eta$ is not a root of $1$. Otherwise, $e$ is the order of $\eta$ as a root of $1$. We assume that 
 $e\neq 1$. 
$\mathcal{H} (\bs)$   is non semisimple in general and 
 its representation theory is usually studied through its decomposition matrix which we now define. 
By the works of Dipper, James and Mathas \cite{DJM}, one can construct 
 a certain set of finite dimensional $\mathcal{H} (\bs)$-modules called {\it Specht modules}:
 $$\left\{ S^{\ulambda} \ |\ \ulambda\vdash_l n\right\}.$$
 The simple $\mathcal{H} (\bs)$-modules  are indexed by a certain set of multipartitions called FLOTW multipartitions $\Uglov{\bs,e} (n)$ (see  \cite[\S 5.7]{GJ})
  $$\left\{ D^{\umu} \ |\ \umu\in\Uglov{\bs,e} (n) \right\}.$$
  If  $e=\infty$, we have $\Uglov{\bs,e} (n)=\Uglov{\bs} (n)$, the set of cylindric multipartitions.
\end{abs}
\begin{abs} 
%
%
%
%
%

If $\ulambda\vdash_l n$  then one can consider the composition multiplicities 
   $[S^{\ulambda}:D^{\umu}]$, with $\umu\in \Uglov{\bs,e} (n)   $. 
    The resulting matrix: 
    $$\mathcal{D}:=([S^{\ulambda}:D^{\umu}])_{\ulambda\vdash_l n,\umu\in \Uglov{\bs,e} (n) } ,  $$
is called the decomposition matrix of $\mathcal{H} (\bs)$. 
The problem of computing the decomposition matrix has been solved by Ariki \cite{ariki} by proving a generalization of a conjecture by Lascoux, Leclerc and Thibon.
 The theorem asserts that the decomposition numbers $[S^{\ulambda}:D^{\umu}]$ with $\umu\in \Uglov{\bs,e} (n)   $ corresponds to the coefficients of the canonical bases for an irreducible highest weight $\Uslaff$-module (realized as submodules of the Fock space) evaluated at $v=1$. 
\end{abs}
In the case $e=+\infty$, this decomposition numbers are thus the polynomials $d_{\ulambda,\umu} (v)$  evaluated at $v=1$. We obtain the following result.
\begin{Th}\label{princ1}
Assume that $e=+\infty$ and that $\ulambda$ and $\umu$ are $l$-partitions of rank $n$ and assume that $\umu$ is cylindric. Then we have:
  $[S^{\ulambda}:D^{\umu}]=0$ unless $\umu\unrhd  \ulambda^R$ 
 while $[S^{{\ulambda}}: D^{\ulambda^R} ]=1$.
\end{Th}
By results of several authors, the polynomials  $d_{\ulambda,\umu} (v)$ also have an interpretation in terms of the representation theory of Ariki-Koike algebras. 
 They correspond to graded decomposition numbers (see \cite{Kl}). Thus Theorem \ref{princ} can be interpreted as a graded analogue of the above regularization Theorem.

\begin{abs}
It is natural to ask what happen in the case where $e\in \mathbb{N}$. Here the main problem is to find 
 a natural order on the set of multipartitions which is the analogue of the dominance order on partitions. A natural choice for it is the one used in \cite{Algo}.
   Now,  the decomposition matrices 
   for Ariki-Koike algebras  can be computed using the algorithm described in this paper and implemented in \cite{Algoprog}. 
   Let $e=2$, $l=2$ , ${\bf s}=(0,1)$ and $n=6$ then for the $2$-partitions   $\ulambda=(3,3)$, 
    $\umu^1=(4,2)$ and $\umu^2=(2,4)$, we have $[S^{\ulambda} : D^{\umu^1}] =[S^{\ulambda} : D^{\umu^2}]=1$ 
     and there are no partition $\unu$ such that $[S^{\ulambda} : D^{\unu}]\neq 0$  which are less than 
      $\umu^1$ and $\umu^2$.  
   Thus,  an analogue of Theorem \ref{princ1}  is not 
    available for these choices. 
\end{abs}

\vspace{1cm}
\noindent {\bf Address}\\
\noindent \textsc{Nicolas Jacon}, UFR Sciences et Techniques,
16 Route de Gray,
25030 Besan\c con
FRANCE\\  \emph{njacon@univ-fcomte.fr}


\begin{thebibliography}{99}                                                                                               

\bibitem{ariki}  \textsc{S. Ariki, }
Representations of quantum algebras and combinatorics of Young tableaux, University Lecture Series, vol. 26, American Mathematical Society, Providence, RI, 2002. Translated from the 2000 Japanese edition and revised by the author.


\bibitem{DJM}  \textsc{R.~Dipper, G.~James, A.~Mathas, }Cyclotomic $q$-Schur
Algebras, Mathematische Zeitschrift 229, (1998), 385-416.

\bibitem{F1}
\textsc{M. Fayers,}
Regularisation and the Mullineux map. Electron. J. Combin. 15 (2008), no. 1, Research Paper 142, 15 pp.

\bibitem{F2}
\textsc{M. Fayers,}
 $q$-analogues of regularisation theorems for linear and projective representations of the symmetric group. J. Algebra 316 (2007), no. 1, 346-367.





\bibitem{GJ}
\textsc{M. Geck and  N. Jacon,}
Representations of Hecke algebras at roots of unity. Algebra and Applications, 15. Springer-Verlag London, Ltd., London, 2011. xii+401 pp.


\bibitem{J}
\textsc{N. Jacon,}
On the parametrization of the simple modules for Ariki-Koike algebras at roots of unity. 
 J. Math. Kyoto Univ. Volume 44, Number 4 (2004), 729-767. 

\bibitem{Algo}
\textsc{N. Jacon,}
An algorithm for the computation of the decomposition matrices for Ariki-Koike algebras. 
J. Algebra 292 (2005), no. 1, 100-109.

\bibitem{Algoprog}
\textsc{N. Jacon,}
GAP program for the computation of the decomposition matrices for Ariki-Koike algebras. 
\url{https://lmb.univ-fcomte.fr/IMG/zip/jacon_arikikoike.g.zip}

\bibitem{Ja}
\textsc{G. James,} On the decomposition matrices of the symmetric groups II.
  J. Algebra 43 (1976),
45-54

\bibitem{Kl}  \textsc{A.~Kleshchev,}
Representation theory of symmetric groups and related Hecke algebras, 
Bull. Amer. Math. Soc. 47 (2010), 419-481 

\bibitem {LLT}\textsc{A.~Lascoux}, B. \textsc{Leclerc and  J-Y Thibon}, 
Hecke algebras at roots of unity and crystal bases of quantum affine algebras.
Comm. Math. Phys., 181 no. 1, 205-263 (1996).



\bibitem{LM}
 \textsc{Leclerc and H.  Miyachi,}
   Constructible characters and canonical bases. J. Algebra, 277 (2004), 298-317.


\bibitem{mat}
\textsc{A. Mathas,}
Iwahori-Hecke algebras and Schur algebras of the symmetric group,
University lecture series, {\bf 15}, American Mathematical Society (Providence, R.I.), 1999.

\end{thebibliography}
\end{document}